\documentclass[11pt]{amsart}
\usepackage[letterpaper, margin=1.5in]{geometry}
\usepackage{amssymb, amsmath, amsthm, amsfonts, mathtools}
\usepackage[english]{babel}
\usepackage{enumitem, comment, graphicx, subcaption, tikz-cd, xcolor}
\usepackage[all]{xy}
\usetikzlibrary{calc,babel}
\usepackage[font=small,labelfont=bf]{caption}
\usepackage[toc,page]{appendix}
\usepackage{float, wrapfig, nicematrix, adjustbox, verbatim, colortbl}
\usepackage{hyperref}
\usepackage{CJKutf8}
\usetikzlibrary{patterns,patterns.meta,decorations.pathreplacing,calc}
\usepackage{scalerel}
\usepackage{csquotes}
\usepackage{mathrsfs}
\usepackage{dutchcal}
\DeclareMathAlphabet{\CMcal}{OMS}{cmsy}{m}{n} 

\theoremstyle{plain}
\newtheorem{theorem}{Theorem}[section]
\newtheorem{lemma}[theorem]{Lemma}

\newtheorem{proposition}[theorem]{Proposition}

\theoremstyle{definition}
\newtheorem{definition}[theorem]{Definition}
\newtheorem{example}[theorem]{Example}

\theoremstyle{remark}
\newtheorem{remark}[theorem]{Remark}

\newcommand{\N}{{\mathbb N}}
\newcommand{\C}{{\mathbb C}}

\newcommand{\X}{{\mathcal X }}

\newcommand{\R}{\mathbb R}

\newcommand{\Z}{\mathbb{Z}}

\newcommand{\fD}{\mathfrak{D}}
\newcommand{\fS}{\mathfrak{S}}

\renewcommand{\log}{\mathrm{log}}

\newcommand{\im}{\mathrm{Im}}

\newcommand{\g}{\mathcal{g}}

\newcommand{\sym}{\mathrm{Sym}^{\!+}}
\newcommand{\Sp}{\mathrm{Sp}}
\newcommand{\cH}{\mathcal{H}}
\newcommand{\tP}{\mathtt{P}}
\newcommand{\tN}{\mathtt{N}}
\newcommand{\tL}{\mathtt{L}}

\newcommand{\SL}{\mathrm{SL}}
\newcommand{\fF}{\mathfrak{F}}

\newcommand{\diam}{\mathrm{diam}}

\newcommand{\sbb}{\text{\begin{CJK}{UTF8}{min}{\scalebox{.7}{\mbox{さ}}}\end{CJK}}BB}
\newcommand{\WP}{\scaleto{W\!P}{3.5pt}}
\newcommand{\trWP}{\scaleto{trWP}{3.5pt}}

\newcommand{\cT}[1]{\mathscr{T}_{#1}^{\scriptscriptstyle{>0}}}
\newcommand{\cTf}[2]{\mathscr{T}_{#1}^{\,{#2}}}

\usepackage[
            backend=biber,
            style=numeric,
            sorting=nyt,
            doi=false,
            url=true,
            isbn=false,
            eprint=false]{biblatex}
\title[Modular interpretation of the Weil--Petersson metric]{Modular interpretation of the Weil--Petersson metric asymptotics for abelian varieties}
\author{Yanbo Fang, Andres Gomez}
\date{}
\begin{document}
\begin{abstract}
   As a first step towards a refined description of the asymptotic of the Weil--Petersson metric on the moduli space of polarized Calabi-Yau manifolds we investigate the concrete case of abelian varieties by linking such asymptotic with the multi-scale collapsing limits of the parametrized flat tori, as explicitly classified by Odaka.   
\end{abstract}
\maketitle
\section{Introduction}
Motivated by the problem of studying the possible \emph{geometric} compactifications of the moduli space of polarized Calabi–Yau manifolds, we explore the relations between the collapsing of the  Ricci flat Kähler metrics along degenerating families of polarized Calabi-Yau manifolds and the asymptotic behavior of the Weil--Petersson metric at infinity on their moduli space. In this paper, we first make progress on this question by investigating the concrete case of principally polarized abelian varieties.

In order to justify our approach, let us elaborate on what we mean by a \emph{geometric compactification}. By Yau's celebrated solution of the Calabi conjecture, for any compact K\"ahler manifold $(X,\omega)$ with $c_1(X)=0$ there exists a \emph{unique} Ricci-flat K\"ahler metric $\omega_{\scriptscriptstyle{CY}}$ in the K\"ahler class $[\omega]$. Therefore, given a polarized degenerating family of Calabi-Yau manifolds $(\mathcal X^*, \mathcal L^*)\rightarrow \Delta^*_t$, as $t\to 0$ we approach the boundary of moduli with a fixed K\"ahler class and the complex structure degenerates. If we denote by $\omega_{\scriptscriptstyle{CY, t}}$ the unique Ricci-flat K\"ahler metric on $X_t$ in the class $2\pi c_1(\mathcal L|_{X_t})$, we say that a compactification is \emph{geometric} if the boundary of the compactified moduli space parametrizes the metric limit of $\omega_{\scriptscriptstyle{CY,t}}$ as $t\to 0$. In the concrete case of principally polarized abelian varieties, where $\omega_{\scriptscriptstyle{CY}}$ is a flat K\"ahler metric on a complex torus, these limits were classified by Odaka in \cite{Odaka2019}. There he explicitly determines the Gromov--Hausdorff limits of principally polarized abelian varieties and provides a moduli theoretic framework to understand such limits.

Following this perspective, the Weil--Petersson metric is a natural candidate for studying the geometry at infinity on the moduli space of polarized Calabi-Yau manifold because, by definition, the Weil--Petersson metric is induced by the variation of the Ricci-flat Kähler metric on the K\"ahler class fixed by the polarization as the complex structure varies (see Section \ref{sec:wp.Ag} for more details); along this lines, it has been proved by Wang \cite{Wang97} and Tosatti \cite{Tosatti15} that on a one-parameter degeneration of polarized Calabi-Yau's the origin lies at finite Weil--Petersson distance if and only if the central fiber of the degeneration is a Calabi-Yau variety with at worst canonical singularities, which implies that the polarized Ricci-flat Kähler metrics have uniformly bounded diameter and are uniformly volume non-collapsed. In contrast to this set-up, the degenerations we are interested are always at infinite Weil--Petersson distance and the Ricci-flat Kähler metric will collapse, phenomena that is already present for abelian varieties as explicitly classified by Odaka.

The main theorem of this paper is the following result that morally says that the asymptotic behavior of the Weil--Petersson metric on the moduli of principally polarized abelian varieties is encoded in the multi-scale collapsing geometry of the (canonical) flat K\"ahler metric along degenerating sequences of abelian varieties. 
\begin{theorem}
\label{thm:main.theorem}
    Let $(A_g,\g_{\WP})$ be the moduli space of complex $g$-dimensional principally polarized abelian varieties endowed with its canonical Kähler structure given by the Weil--Petersson metric. Then for every non-degenerate direction at infinity in $A_g$, there exists an integer $g'$ with $0\leq g'<g$ such that the Weil--Petersson metric converges in the pointed Gromov--Hausdorff sense, to a Weil--Petersson type metric on
    \begin{equation}
        (A_{g'}\times \cT{g-g'},\, \g'_{\WP}+ \g''_{\scaleto{trWP}{3.5pt}}),
    \end{equation}
    where
    \begin{itemize}[leftmargin=2em]
         \item the moduli space $A_{g'}$ is naturally identified with a boundary stratum of the constant-volume Gromov--Hausdorff compactification $\bar{A}_g^{\mathrm{GH}_{\mathrm{vol}}}$ of $A_g$, and $\g'_{\WP}$ denotes the usual Weil--Petersson metric on $A_{g'}$;
         
         \item $\cT{g-g'}$ denotes the moduli space of real $(g-g')$-dimensional flat tori of arbitrary positive diameter, which at any fixed diameter $\cTf{g-g'}{c}$ is naturally realized as a boundary stratum of the constant-diameter Gromov--Hausdorff compactification $\bar{A}_g^{\mathrm{GH}_{\mathrm{diam}}}$ of $A_g$. Furthermore, $\g''_{\scaleto{trWP}{3.5pt}}$ is a ``tropical Weil--Petersson type'' metric on $\cT{g-g'}$, defined via the variation of the limiting flat metrics.
    \end{itemize}
\end{theorem}

A few remarks are in order. 

\begin{remark}
    Moduli spaces like $A_g$, that can be realized as locally symmetric spaces, naturally admit several distinct compactification frameworks. Consequently, identifying which of these compactifications carries a modular geometric interpretation is one of the guiding questions of our investigation. Theorem \ref{thm:main.theorem} suggest that the Weil--Petersson metric can indeed shed light on this problem: as the abelian varieties degenerate, their canonical flat metrics collapse, and this collapse is reflected in the asymptotic geometry of the Weil–Petersson metric.
    
\end{remark}

\begin{remark}
    The theorem's formulation requires a precise notion of a \emph{non-degenerate direction at infinity in $A_g$}, which we give in Definition~\ref{def:non.deg.direc.Ag}. Roughly speaking this says that, given a sequence $\tau_n\in A_g$ after choosing lifts $\tau_n=X_n+iY_n\in\mathfrak{S}_g$ contained in a Siegel set, one may assume that $Y_n$ has exactly $g-g'$ eigenvalues tending to infinity, while the remaining $g'$ eigenvalues stay uniformly bounded above and below. Non-degeneracy then amounts to ruling out the possibility that a sub-collection of the diverging eigenvalues separates into different scales, which would force an additional splitting and hence a lower-rank limit. (Cf.\ Theorem~3.1 in \cite{Odaka2019} for an algebraic characterization of this condition).
\end{remark}
 
\begin{remark}
    One immediate consequence of Theorem \ref{thm:main.theorem} is that along any such non-degenerate direction, $A_g$ collapses in the pointed Gromov--Hausdorff sense to the lower--dimen\-sional space $A_{g'}\times\cT{g-g'}$. From a Riemannian viewpoint, this behavior is expected given that it is classical that $A_g$ is a complete non-compact Riemannian orbifold with finite volume.
\end{remark}

The organization of the paper is as follows. In \S2, we introduce the necessary preliminary results and fix the notation and conventions used throughout. We first recall, in \S2.1, the basic theory of moduli of complex abelian varieties and fix conventions for the Siegel upper half space together with the relevant reduction theory. Next, in \S2.2, in the setting of the Siegel upper half-space, we recall the Lie theoretic notion of \emph{boundary component} for a Hermitian symmetric domain and the algebraic and differential-geometric data naturally associated with it. Next, in \S2.3 we revisit the work of \cite{Odaka2019,OO2021} on the Gromov--Hausdorff compactifications of $A_g$. The core of the paper is \S3. In \S3.1 we recall the definition of the Weil--Petersson metric and the standard derivation of its explicit formula on $A_g$, which serves as the starting point for our analysis in the remainder of the paper. Finally In \S3.1 we prove the main theorem after developing the necessary technical preliminaries. 

\textbf{Acknowledgments}. We are deeply grateful to Cristiano Spotti for suggesting this problem and for his constant support, insightful advice, and many stimulating discussions throughout the development of this work. We would also like to express our gratitude to Yuji Odaka, whose pioneering ideas have greatly influenced our investigation; the second author would also like to thank him for his hospitality during a visit to Kyoto University. Finally, the second author thanks Masafumi Hattori for helpful comments and discussions. The authors thank the Villum Foundation for its support through the grant Villum YIP+ 00053062.

\section{Preliminaries}

\subsection{Moduli space of principally polarized abelian varieties.}
\label{mppav}
In this section we recall some classical notions about complex abelian varieties and set up our notation on the Siegel upper half space and the Siegel reduction theory, we follow closely \cite{Chai84}. For a detailed treatment and all of the proofs, we refer to \cite{Namikawa80, Hulek1993}.

Recall that $A_g$, the moduli space of principally polarized abelian varieties  of complex dimension $g$,  is given the structure of complex analytic space by means of the identification
\[
A_g\simeq\Gamma\backslash \fS_g,
\]
where $\fS_g$ is the \emph{Siegel upper half-space}, namely $\fS_g=\{\tau\in \mathrm{Sym}(g, \C)| \,\im(\tau)>0\}$ and $\Gamma=\Sp(2g,\Z)/\{\pm1\}$ is an arithmetic subgroup of the group $\Sp(2,\R)/\{\pm1\}$ of holomorphic automorphism of $\fS_g$, whose action is given by
\[
M=\begin{pmatrix}
    A&B\\
    C&D
\end{pmatrix}:\tau\to M\cdot\tau=(A\tau+B)(C\tau+D)^{-1}, \quad \tau\in\fS_g, M\in\mathcal{G}.
\]

Notice that, as an analytic space, $A_g$ has the structure of a complex orbifold, or V-manifold, due to the presence of abelian varieties with extra automorphism.  

\begin{definition}
Let $n\ge 3$ and let $\Gamma(n)$ denote the principal congruence subgroup of level $n$,
\begin{equation}
\Gamma(n):=\{\,M\in \Sp(2g,\Z)\mid M\equiv I_{2g}\ (\mathrm{mod}\ n)\,\}.
\end{equation}
We define the moduli space of complex $g$-dimensional principally polarized abelian varieties with level-$n$ structure as the quotient
\begin{equation}
A_g(n):=\Gamma(n)\backslash \fS_g .
\end{equation}
For $n\ge 3$, the group $\Gamma(n)$ is torsion-free; hence the action on $\fS_g$ is free and properly discontinuous, and $A_g(n)$ is a smooth complex manifold. Equivalently, the finite cover $A_g(n)\to A_g$ provides uniformizing charts near the orbifold points of $A_g$.
\end{definition}

We now recall some basic notions from Siegel reduction theory that will be used throughout. For a point $\tau=X+iY\in \fS_g$ we denote the \textit{Jacobi decomposition} of $Y$ as $Y=LDL^t$ where 
\[
L=\begin{pmatrix}
    1      &            &        &        & \\
    l_{12} & 1          &        & \text{\Large{0}}& \\
    l_{13} & l_{23}     & 1      &        & \\
    \vdots & \vdots     & \ddots & \ddots & \\
    l_{1g} & l_{2g}     & \cdots & l_{g-1,g} & 1
\end{pmatrix},\quad D=\mathrm{diag}(d_1,\cdots,d_g).
\]

With respect to this decomposition we have that a Siegel sets $\fF_g(u)$ is be given by
\begin{equation}
    \label{eq:sieg.set.def}
    \fF_g(u)=\{ X+ i\, Y \in \fS_g \mid |x_{ij}|<u, |l_{ij}|<u,\, 1<ud_1,\, d_i<ud_{i+1} \text{ for all } i,j  \}.
\end{equation}

Additionally, it is straightforward to check that a subset of the inequalities in \eqref{eq:sieg.set.def} define the Siegel sets  on $\sym(g,\R)$ for the action of the discrete subgroup $\mathrm{GL}(g,\Z)$. More specifically we have that with respect to the Jacobi decomposition $Y=LDL^t$ for a point $Y\in \sym(g,\R)$ a Siegel set $\fF_g^{+}(u)$ is given by
\begin{equation}
    \label{eq:sieg.set.sym}
    \fF^{+}_g(u)=\{Y \in \sym(g,\R) \mid  |l_{ij}|<u,\, 1<ud_1,\, d_i<ud_{i+1} \text{ for all } i,j  \}.
\end{equation}
Expression that will be needed later in our analysis.

Finally let us recall that for a suitable choice of $u\gg1$ the Siegel set $\fF_g(u)$ is a \textit{coarse fundamental domain} for the action of $\Gamma$ in the sense that they satisfy the following properties:
\begin{enumerate}[label=\roman*.]
    \item $\bigcup_{u > 0} \fF_g(u) = \fS_g$.
    \item $\Gamma \cdot \fF_g(u) = \fS_g$ if $u$ is large enough.
    \item For all $u > 0$ the set $\{ \gamma \in \Gamma \mid \gamma \,\fF_g(u) \cap \fF_g(u) \neq \varnothing \}$ is finite.
\end{enumerate}
Analogously, $\fF_g^+(u)$ will be a coarse fundamental domain for the action of $\mathrm{GL}(g,\Z)$.

\subsection{Boundary components of \texorpdfstring{$\fS_g$}.}
\label{subsec:boun.comp.Sg}In this subsection we introduce the Lie-theoretic concepts underlying the notion of boundary component for a Hermitian symmetric domain. Our main goal is to describe the boundary components of the Siegel upper half space and the
algebraic data naturally attached to them. As this material is technical and notation-heavy, we do not provide a comprehensive exposition but introduce only the necessary notions and results that we will need in later sections. For a detailed treatment and all of the proofs, we refer to \cite{Namikawa80}.

Recall that as a Hermitian symmetric space, the Siegel upper half-space is given by $\fS_g=Sp(2g,\R)/\mathrm{U}(g)$. Here, the unitary group $\mathrm{U}(g)$, which is the stabilizer of the basepoint $iI_g\in \fS_g$, is embedded in $\Sp(2g,\R)$ by $A+iB \to \begin{pmatrix} A&B \\ -B&A \end{pmatrix}$. The Harish-Chandra embedding
\[
\Phi:\fS_g \to \fD_g,\quad\Phi(\tau)=(\tau-iI_g)(\tau+iI_g)^{-1},
\]
gives a realization of $\fS_g$ as a bounded domain $\fD_g$ in $\mathrm{Sym}_g(\C)\simeq \C^{g(g+1)/2}$, where
\[
\fD_g=\{Z\in \mathrm{Sym}_g(\C)\, |\, I_g -Z\bar Z>0 \}.
\]
The boundary components of $\fS_g$ are defined as an stratification of the topological closure $\overline\fD_g =\{Z\in \mathrm{Sym}_g(\C)\, |\, I_g -Z\bar Z\geq 0\}$. A distinguished collection of proper boundary components is obtained by fixing $0\leq g' < g$ and requiring
\[
\bigg\{ Z\in \overline\fD_g \,\bigg|\, Z =\begin{pmatrix} Z'&0\\0&I_{g-g'}\end{pmatrix}, \text{ and } Z'\in\fD_{g'} \bigg\}.
\]
Furthermore, each boundary component is a $\Sp(2g,\R)$-translate of one of the \emph{distinguished} boundary components. It is clear that each \emph{distinguished} boundary component can be identified with $\fD_{g'}$ for $0\leq g' < g$, and by abuse of notation we identify them with their corresponding unbounded realization $\fS_{g'}$. 

Associated with each boundary component $\fS_{g'}$ we have its normalizer subgroup $P_{g'}$ in $\Sp(2g,\R)$. Each $P_{g'}$ is a maximal real parabolic subgroup and can be realized as the normalizer of the isotropic-co-isotropic flag $I\subset I^{\perp}$, where $I=\langle e_{g'+1},\dots,e_g\rangle$. Let us fix $g=g'+g''$ throughout, a direct computation will show that $P_{g'}$ consist of matrices 
\[
P_{g'} =
    \left\{ 
        \begin{pmatrix}
        A'& 0 & B' & * \\
        * & u & * & *\\
        C' & 0 & D' & * \\
        0 & 0 & 0 & (u^t)^{-1}
        \end{pmatrix}\!\in \Sp(2g,\R)\, 
        \middle|
        \begin{pmatrix}
            A'&B'\\
            C'&D' 
        \end{pmatrix}\!\in \Sp(2g',\R),\, u \in \mathrm{GL}(g'',\R)
    \right\}.
\]
Moreover, relative to the filtration $I\subset I^{\perp}$ we have the Levi decomposition $P_{g'}=N_{g'}\rtimes L_{g'}$, where $L_{g'}$ is the Levi factor given by the direct product $L=\Sp(I^{\perp}/I)\,\mathrm{GL}(I)=\Sp(2g',\R)\,\mathrm{GL}(g'',\R)$, and $N_{g'}$ is the unipotent radical of $P_{g'}$ which is given in terms of the matrices
\[
N_{g'}=
    \left\{ 
        \begin{pmatrix}
        I_{g'} & 0      & 0     & n \\
        m^t    & I_{g''}& n^t   & b \\
        0      & 0      & I_{g'}& -m\\
        0      & 0      & 0     & I_{g''}
        \end{pmatrix}\, 
        \middle|\,
        n^t m + b = m^t n + b^t
    \right\}.
\]
Finally, following standard terminology we will reefer to $L_{g'h}=\Sp(2g',\R)$ as the Hermitian factor and to $L_{g'l}=\mathrm{GL}(g'',\R)$ as the linear factor. 

From a more geometric perspective it is worth noting that, as differentiable manifolds, the decomposition $P_{g'}\cong N_{g'}\,L_{g'h}\,L_{g'l}$ is a diffeomorphism. Furthermore, this decomposition gives the existence of a diffeomorphism
\begin{equation}
    \label{eq:lang.decom.fs}
    \fS_g\cong P_{g'}/(\mathrm{U}(g)\cap P_{g'})\cong N_{g'}\times \fS_{g'}\times \sym(g'',\R),
\end{equation}
which we call the \emph{horospherical} decomposition of $\fS_g$ relative to $P_{g'}$. To see this we first need to recall that $P_{g'}$ acts transitively on $\fS_{g'}$ and that the Levi factor is stable under the Cartan involution corresponding to $\mathrm{U}(g)$. Then we observe that 
\[
    \mathrm{U}(g)\cap P_{g'}= \left\{ \begin{pmatrix}
                                            A'& 0 & B' & 0 \\
                                            0 & u & 0 & 0\\
                                            -B' & 0 & A' & 0 \\
                                            0 & 0 & 0 & u
                                        \end{pmatrix}
    \in \Sp(2g,\R)\, \middle|\, A'+iB' \in \mathrm{U}(g'),\, u \in \mathrm{O}(g'',\R) \right\},
\]
which identifies the boundary component $\fS_{g'}$ with the symmetric space of $L_{g'h}=\Sp(2g',\R)$ and the cone $\sym(g'',\R)$ with the symmetric space of $L_{g'l}=\mathrm{GL}(g'',\R)$. Finally, with respect to this decomposition we define the map 
\begin{equation}
    \label{eq:pi.gp.map}
    \pi_{g'}:\fS_g\to \fS_{g'}\!\times\sym(g'',\R), \,\tau =\!
\begin{pmatrix}
\tau' & \tau''' \\
^t\tau''' & \tau''
\end{pmatrix} \longmapsto (\tau',Y''-{}^tY'''(Y')^{-1}Y''').
\end{equation}

The geometric significance of the decomposition of Equation \eqref{eq:lang.decom.fs} is that it defines a coordinate system of the quotient $\Gamma \backslash \fS_g$ adapted to the boundary component $\fS_{g'}$. To see this we have the following result that essentially tell us that for a \textit{suitable neighborhood} of $\Gamma\backslash\fS_{g}$ the action of $\Gamma$ is equivalent to the action of $P_{g'}\cap\Gamma$.

\begin{proposition}
    \label{prop:inv.neigh}
     Fix $0\leq g' <g$, and let $\mathtt{P}_{g'}:=\Gamma\cap \tP_{g'}$. Then, there exist a $\tP_{g'}$-invariant neighborhood $\mathcal{U}_{g'}$ of $\fS_{g'}$ in $\fS_{g}$ such that the inclusion $\mathcal{U}_{g'}\hookrightarrow \fS_{g}$ induces a holomorphic embedding 
    \[
    \tP_{g'} \backslash \mathcal{U}_{g'}\hookrightarrow \Gamma\backslash\fS_{g}.
    \]
\end{proposition}

\begin{proof}
The existence of such a $\mathtt{P}_{g'}$-invariant neighborhood is a consequence of the Piatetskii-Shapiro cylindrical topology (see \cite{AMRT}, Chap.~III, \S 1.). 
\end{proof}

In practice, we can described such invariant neighborhood in terms of the Siegel sets as follows. Fix a $u>0$ such that $(\Gamma \cap L_{g'h}) \cdot \fF_{g'}(u) = \fS_{g'}$, for all $ 1 \leq g' \leq g$ and for $r>0$ sufficiently large, set
\begin{equation}
    \mathcal{U}_{g'}^r:=\left\{\tau\in\fF_g(u)\,|\,d_{g'+1}>r\right\}.
\end{equation}
This will define a family of $\tP_{g'}$-invariant neighborhoods of $\fF_{g'}(u)\subset\fS_{g'}$. 
Moreover, we notice that if we have a sequence $\{\tau_n\}_{n\in\N}$ in $\fF_g(u)$ we conclude that $\tau_n\to\tau'$, for some $\tau'\in \fF_{g'}(u) $ if and only if   
$\pi_{g'}(\tau_n)\to\tau'$ and $d_{g'+1}(\tau_n)\to \infty$. 

The relevance of this result to our analysis is that to study the asymptotic behavior of the Weil--Petersson metric near the boundary components of $\Gamma\backslash\fS_g$ we can instead work on the neighborhoods $\mathcal{U}_{g'}$ modulo the action of the simpler subgroup $\mathtt{P}_{g'}$. Before we proceed, let us introduce the following notation that we will use throughout the paper.

\textbf{Notation:} Fix $0\leq g'< g$, and let $P_{g'}$ be the parabolic subgroup associated with the $\fS_{g'}$ boundary component of $\fS_g$. Then relative to the arithmetic subgroup $\Gamma$ we denote by 
\[
\tP_{g'}=\Gamma\cap P_{g'},\quad \mathtt{N}_{g'}:=\Gamma \cap N_{g},\quad \tL_{g'}=\Gamma\cap L_{g'} \footnote{Notice, that given a general Levi decomposition $P=NL$, $\tL$ is in principly only a finite index subgroup of $\im(\tP\to L)$. However, given that we have fixed the data of the isotropic-co-siotropic flag the Levi decomposition is such that the integral structure of $\tL$ agrees with the one of $\tP$}.
\]
Moreover, for any congruence subgroup $\Gamma(n)$ we have the analogous notation, i.e., $\tP_{g'}(n),\tN_{g'}(n)$, and $\tL_{g'}(n)$.

To understand the quotient $\tP_{g'}\backslash \fS_g$, notice that $N_{g'}$ is a normal subgroup of $P_{g'}$. Thus, we have the exact sequence 
\begin{equation}\label{UPquot}
    1\to N_{g'}\to P_{g'}\to L_{g'}\to 1.
\end{equation}
This suggest that we can compute the quotient $A_{g}\simeq \Gamma\backslash\fS_g$ in two steps. First we take the quotient by $\mathtt{N}_{g'}$, which only acts on the $N_{g'}$ factor. By general theory of arithmetic groups \cite{BHC62} it follows that $\mathtt{N}_{g'}$ is a uniform lattice in $N_{g'}$ and consequently the quotient $N_{g'}/\mathtt{N}_{g'}$ is a compact nilmanifold. Thus, using the decomposition in Equation \eqref{eq:lang.decom.fs} and the map defined in Equation \eqref{eq:pi.gp.map} we get the map
\[
\tN_{g'}\backslash\fS_g\longrightarrow \fS_{g'}\times \sym(g'',\R),
\]
that realizes $\tN_{g'}\backslash\fS_g$ as a trivial fiber bundle over the base $\fS_{g'}\times\sym(g-g',\R)$ and with fibers the compact nilmanifold $\tN_{g'}\backslash N_{g'}$. Next, we take the quotient by $\tL_{g'}$. This has two effects: it replaces the base of the fiber bundle by
\[
\tL_{g'}\backslash\big(\fS_{g'}\!\times\!\sym(g'',\R)\big)\cong \big(Sp(2g',\Z)\backslash\fS_{g'}\big)\!\times\!\big( \mathrm{GL}(g'',\Z)\backslash \sym(g'',\R)\big)
\]
and on the fibers it acts with the stabilizer subgroup of the base-point. Thus, given that in general $\Gamma$ is not torsion-free we might get a flat orbifold structure in the fibers. Now, since in Section \ref{sec:asymp.wpm} we will work with the quotient
\[
\Gamma\backslash\mathcal{U}_{g'}\simeq\tP_{g'}\backslash\mathcal{U}_{g'}\subset\tP_{g'}\backslash\fS_g,
\]
viewed as a Riemannian orbifold endowed with the Weil--Petersson metric, local arguments must a priori be carried out in orbifold charts. To avoid repeatedly invoking this language, we will carry out most of our analysis on the Siegel sets i.e., we will consider instead the map 
\begin{equation}
    \pi_{g'}:\fF_{g}(u)\longrightarrow \fF_{g'}(u)\times\fF_{g''}^+(u)
\end{equation}
with $u\gg 1$ suitably chosen. 

\subsection{Geometric compactifications of \texorpdfstring{$A_g$}.}
\label{sect:geom.comp.Ag}
There are several important compactification frameworks for modular varieties such as $\Gamma\backslash \fS_g$, which are arithmetic quotients of Hermitian symmetric domains. In this section we review two compactifications: the Satake--Baily--Borel (SBB) compactification and the tropical compactification, in the sense of Odaka. We place special emphasis on their differential-geometric interpretation, as developed in \cite{Odaka2019}.

The Satake-Baily-Borel compactification $\bar{A_g}^{\sbb}$ \,\footnote{Following \cite{Odaka2019} and \cite{Namikawa80} we use the \textit{Hiragana} type character ``\begin{CJK}{UTF8}{min}{\scalebox{.7}{\mbox{さ}}}\end{CJK}'' which is pronouced as ``SA'', the first syllable of Satake. we refer to the footnote in \cite{Odaka2019} for a brief explanatory note on the choice of notation.} is defined by the following construction. First, let us define the partial compactification
\[
\fS_g^*:=\fS_g\,\sqcup\,\fS_{g-1}\,\sqcup\,\cdots\sqcup\,\fS_0,
\]
obtained by adjoining the distinguished rational boundary components. This set $\fS_g^*$ is endowed with the Satake topology (see  \cite{AMRT} Chapter III \S $6.1$ for details). The Satake-Baily-Borel compactification is then defined as the quotient $ \bar{A_g}^{\sbb}:=\fS_g^*/\!\sim$,
where $\sim$ is the equivalence relation extending the action of $\Gamma$ on $\fS_g$.

More recently, in \cite{Odaka2019}, with the aim of understanding the relationship between the differential-geometric and the algebro-geometric compactifications of moduli, Odaka explicitly classifies the Gromov--Hausdorff limits of the flat metric along degenerating sequences of complex principally polarized abelian varieties and studies the geometry of the resulting Gromov--Hausdorff compactifications of $A_g$. In particular, the following connection between the Satake-Baily-Borel compactification and the Gromov--Hausdorff compactification at the volume scale is established.

\begin{theorem}[\cite{Odaka2019} Proposition $2.13$]
\label{thm:vol.GH.limits.abelian}
Fix $0\leq g'<g$ and consider a sequence of pointed $g$-dimensional complex principally polarized abelian varieties $(X_n,p_n)$ which is converging to a point of $A_{g'} \subset \partial\bar{A_{g}}^{\sbb}$. Then, the pointed Gromov--Hausdorff limit of  $(X_n,p_n)$ with fixed volume $1$ is isometric to 
\[
(X_\infty\times \mathbb{R}^{g''},(p_\infty,0)),
\] 
where $X_\infty$ is the real $2g'$-dimensional flat torus underlying a complex $g'$-dimensional principally polarized abelian variety and the $\mathbb{R}^{g''}$ is given the flat metric. 
\end{theorem}

Therefore, as an immediate consequence of this result one obtains that the Satake-Baily-Borel compactification $\bar{A_g}^{\sbb}$ parametrizes the set of pointed Gromov-- Haussdorﬀ limits of complex $g$-dimensional principally polarized abelian varieties with fixed volume, or in other words that this compactification can be differential--geometrically recovered as parametrizing the volume scale limits of the \emph{canonical} flat K\"ahler metric underlying a principally polarized abelian varieties.

On the other hand, the Gromov--Hausdorff limits of principally polarized abelian varieties at the diameter scale are classified as follows.
\begin{theorem}[\cite{Odaka2019} Theorem $2.3$]\label{AV.lim.max}
Fix $0\le g'<g$, and let $X_n$ be any sequence of  complex $g$-dimensional principally polarized abelian varieties converging to a point of $A_{g'}\subset \partial\bar A_g^{\sbb}$. After rescaling so that $\operatorname{diam}(X_n)=1$, the sequence $X_n$ converges, in the Gromov--Hausdorff sense, to a $(g-r)$-dimensional real flat torus of diameter $1$, for some $r$ satisfying $g'\le r<g$. Conversely, every $(g-r)$-dimensional real flat torus of diameter $1$ with $g'\le r<g$ occurs as the Gromov--Hausdorff limit of such a rescaled sequence.
\end{theorem}

Odaka goes on to construct $\bar{A_g}^{T}$, which he calls the \emph{tropical geometric compactification} of the moduli space of principally polarized abelian varieties, as the Gromov--Hausdorff compactification of $A_g$ at the diameter scale. As a topological space $\bar{A_g}^T$ is obtained by attaching the moduli space of real flat tori with diameter 1 of dimension $1\leq r \leq g$, 
\begin{equation}
    \label{eq:trop.comp.abel}
    \bar{A_g}^T:=A_g\,\bigsqcup_r\,\mathscr{T}_{r},\qquad\text{where }\,\mathscr{T}_r:=\mathrm{GL}(r,\Z)\backslash\sym(r,\R)/\R_{>0},
\end{equation}
and putting the topology whose basis induces the complex analytic topology in $A_g$ and the Gromov--Hausdorff topology on the boundary. 

One should emphasize that this compactification is no longer an algebraic variety. From our perspective, however, this is a feature rather than a drawback: it is the geometric constraint imposed by requiring the boundary to parametrize limits of the canonical flat K\"ahler metric on a principally polarized abelian variety, which collapse to real manifolds at diameter scale. Furthermore, the resulting limits have a very interesting connection with the algebro-geometric degenerations of principally polarized abelian varieties, as explained by the following result of Odaka.

\begin{theorem}[\cite{Odaka2019} Theorem $3.1$]\label{thm:Oda.holo.family}
Let $\pi^*: (\X^*,\mathcal{L}^*)\to \Delta^*_t$ be a flat projective family of $g$-dimensional principally polarized abelian varieties which extends to some quasi-projective family over $\Delta_t$. Furthermore, we assume that we have semi-abelian reduction over $0\in \Delta_t$ and that the Raynaud extension is the trivial extension. Let $t_n \in \Delta^*_t$ be an arbitrary sequence converging to $0$. Then, the fiber $\X_{t_n}$ with the rescaled flat K\"ahler metric of diameter 1 collapses to a $r$-dimensional real torus where $r$ is the torus rank of $\X_0$. In Particular the collapsing limit is homeomorphic to the dual complex of the degeneration.
\end{theorem}

Notice that the result above says that, if we fix $0\leq g' <g$ and $X_n\in A_g$ is sequence converging to a point $A_{g'}\subset \partial \bar A_g^{\sbb}$ along an holomorphic family, the dimension $r=\dim(X_\infty)$ of the real flat torus appearing as the diameter-rescaled limit is completely determined by the corank $r=g-g'$ of the boundary component $A_{g-r}$. This contrasts with Theorem \ref{AV.lim.max}, which allows the dimension to satisfy $g'\leq r<g$. 

As a consequence of this result, we obtain a correspondence between the boundary strata of the Satake--Baily--Borel compactification and the tropical compactification; more precisely, we relate the differential-geometric data encoded by the volume-scale and diameter-scale Gromov--Hausdorff compactifications of $A_g$ (see the diagram below).
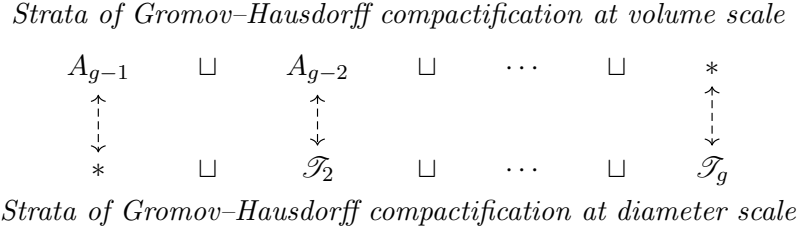
\[
\begin{array}{c}
\textit{Strata of Gromov--Hausdorff compactification at volume scale}\\[0.4em]
\begin{tikzcd}[column sep=1.5em]
A_{g-1} \arrow[d, leftrightarrow, dashed] & \sqcup &
A_{g-2} \arrow[d, leftrightarrow, dashed] & \sqcup &
\cdots & \sqcup &
* \arrow[d, leftrightarrow, dashed]\\
* & \sqcup & \mathscr{T}_{2} & \sqcup & \cdots & \sqcup & \mathscr{T}_{g}
\end{tikzcd}\\[0.4em]
\textit{Strata of Gromov--Hausdorff compactification at diameter scale}\\[0.4em]
\end{array}
\]
This, as Odaka remarks, is a manifestation of the tropicalization phenomenon where the degeneration order is reversed when passing from the algebro-geometric setting to the tropical one. 

\begin{remark}\label{rmk:sbb.trop.strata}
Notice that \emph{a posteriori} one can argue that the duality between the strata of the diameter-scale and volume-scale Gromov--Hausdorff compactifications (equivalently, between the Satake--Baily--Borel and the tropical compactifications) agrees with the Lie-theoretic data of the horospherical decomposition
\[
\fS_g \cong N_{g'} \times \fS_{g'} \times \sym(g-g',\R),
\]
relative to the parabolic subgroup $P_{g'}$ associated with the boundary component $A_{g'}$ to which the sequence $X_n$ converges. At this point we do not claim a deeper structural statement, as the duality was derived solely from the differential-geometric behavior of the canonical flat metrics under degeneration.
\end{remark}
  
\begin{remark}
Odaka further shows that the boundary of $\bar A_g^{T}$ can be described as the dual complex of \emph{any}
toroidal compactification of $A_g$, by identifying $\bar A_g^{T}$ with what he calls the generalized
Morgan--Shalen compactification of $A_g$, denoted $\bar A_g^{\scriptscriptstyle\mathrm{MSBJ}}$
(see the appendix of \cite{Odaka2019} for details). Morally, this provides an invariant framework for
compactifying locally symmetric spaces by attaching the data of the essential skeleton/dual complex
associated with a toroidal compactification. Finally, as shown in \cite{OO2021}, the same compactification
$\bar A_g^{T}$ also recovers one of the compactifications proposed by Satake, namely the Satake
compactification for the adjoint representation, denoted $\bar A_g^{\mathrm{Sat},\tau_{\mathrm{ad}}}$
(see Section~2.2 of \cite{OO2021}).
\end{remark}

\section{Asymptotic behavior of the Weil--Petersson metric}
\label{sec:asymp.wpm}
\subsection{The Weil--Petersson metric on \texorpdfstring{$A_g$}.}
\label{sec:wp.Ag}
We recall the definition of the Weil – Petersson metric and briefly review the standard derivation of its explicit expression on $A_g$, which will be used throughout the sequel. 

We work on the deformation space of a polarized Calabi--Yau manifold. Let $\pi:(\X,\mathcal{L})\to S$ be a proper holomorphic polarized family of Calabi--Yau manifolds of dimension $n$. For each $s\in S$, let $\omega_s$ denote the unique Ricci-flat K\"ahler metric on $X_s$ in the class $2\pi c_1(\mathcal{L}|_{X_s})$. Given $v\in T_sS$, let $\rho_s(v)\in H^1(X_s,T_{X_s})\cong H^{0,1}_{\bar\partial}(X_s,T^{1,0}_{X_s})$ be the Kodaira--Spencer class, and denote by $\mu_v$ its $\omega_s$-harmonic representative. The Weil--Petersson metric is the $L^2$ pairing
\[
\g_{\WP}(v_s,\overline{w}_s) = \int_{X_s}\langle \mu_v,\mu_w\rangle_{\omega_s}\,dV_{\omega_s},
\qquad v,w\in T_sS.
\]
From this, it follows that the Weil--Petersson metric is induced by the variation of the Ricci-flat K\"ahler metric on the corresponding fibers: the norm on the tangent space is given by the norm of the harmonic representatives of the Kodaira-Spencer classes with respect to the Ricci-flat K\"ahler metric on the corresponding fibers according to the Calabi-Yau theorem. In \cite{Tian87,Todorov89} it is shown that $\g_{\WP}$ is indeed a K\"ahler metric and that it has negative sectional curvature.

A useful alternative description is in terms of a K\"ahler potential. As shown in \cite{Tian87}, one has 
\[
\omega_{\scaleto{W\!P}{3.5pt}}
=-\,i\,\partial\bar{\partial}\,
\log \int_{X_s} i^{\,n^2}\,\Omega_s\wedge \overline{\Omega_s},
\]
Here, $\Omega_s$ is a holomorphic section of the relative canonical bundle $K_{\X/S}$, and $\omega_{\scaleto{W\!P}{3.5pt}}$ is the Ricci-form associated with $\g_{\WP}$.

\begin{remark}
This expression identifies $\omega_{\mathrm{WP}}$ with the curvature form of the $L^2$ metric on the Hodge line bundle $\lambda=\pi_*\omega_{\X/S}$.   \end{remark}

We now turn to the case of abelian varieties, where the preceding expression becomes explicit.

\begin{proposition}[cf. Schumacher \cite{SCH85}]
On $A_g$ the Weil--Petersson metric is given in terms of the left-invariant metric on $\fS_g$, as follows
\begin{equation}
    \label{eq:wp.metric}
    \omega_{\scaleto{W\!P}{3.5pt}}=\frac{i}{4}\mathrm{tr}\big(Y^{-1}d\tau Y^{-1}d\bar\tau \big).
\end{equation}
In particular, for abelian varieties the Weil--Petersson metric is determined entirely by the locally Hermitian
symmetric space data of $A_g$.
\end{proposition}
\begin{proof}
    We work on the universal covering space $\fS_g$. There, we have that locally an holomorphic family is given by $\X\to \fS_g$, where $\X:=(\fS_g\times \C^g)/\Z^{2g}$ and $\Z^{2g}$ acts on $\fS\times \C^g$ via 
    \[
    \begin{pmatrix}
        n_1\\
        n_2
    \end{pmatrix}: (\tau,v)\mapsto(\tau, v+ n_1 + \tau\cdot n_2 ). 
    \]
    Equivalently, the torus $X_\tau \cong\R^{2g}/\Z^{2g}$, with differentiable coordinates $(x^i,y^i)$, corresponding to the modular point $\tau\in \fS_g$ has holomorphic coordinates given by the map $(x^i,y^i)\mapsto z^i=x^i + \tau^i_j\,y^j$. 
    
    Thus, an holomorphic section $\Omega_\tau$ of the relative canonical bundle is given by means of the standard translation-invariant holomorphic $g$-form
    \[
    \Omega_\tau=dz^1\wedge\cdots\wedge dz^g.
    \]
    Then, a determinant computation yields
    \[
    i^{g^2}\,\Omega_\tau\wedge\overline\Omega_\tau=2^g\,\det(\im(\tau))\; dx^1\wedge\cdots\wedge dx^g\wedge dy^1\wedge\cdots\wedge dy^g,
    \]
    and consequently,
    \[
    \int_{X_\tau} i^{g^2}\,\Omega_\tau\wedge\overline\Omega_\tau=2^g\,\det(\im(\tau)).
    \]
    Finally, a straightforward computation using Jacobi's formula gives
    \[
    \omega_{\scaleto{W\!P}{3.5pt}}=\frac{i}{4}\mathrm{tr}\big(Y^{-1}d\tau Y^{-1}d\bar\tau \big).
    \]\end{proof}

We conclude with two examples that will be revisited in the next section and will serve to motivate the general framework and results.
\begin{example}
    \label{ex:wp.ellip}
    We now take up the simplest instance of our construction, namely the moduli space of elliptic curves. Let
    \[
      A_1 \;=\; \Gamma\backslash \cH,
      \qquad
      \cH=\{\tau=x+iy\in\C \mid y>0\},
      \qquad
      \Gamma=\SL(2,\Z).
    \]
    Denote by $*_\infty\in A_1$ the cusp at infinity. It is well known that a fundamental system of neighborhoods of $*_\infty$ is given, for $d\gg 1$, by
    \[
      \mathcal U_d \;=\; \{\tau\in\cH \mid \im(\tau)>d\}.
    \]
    Let $P_\infty\subset \SL(2,\R)$ be the parabolic subgroup stabilizing $*_\infty$, namely
    \[
      P_\infty \;=\;
      \left\{
        \begin{pmatrix} a & b \\ 0 & a^{-1}\end{pmatrix}
        \;\middle|\;
        a\in\R^\times,\ b\in\R
      \right\},
    \]
    and set $\tP_\infty \;:=\; \Gamma\cap P_\infty$. Then $\tP_\infty$ is generated by the translation
    \[
      T=\begin{pmatrix}1&1\\0&1\end{pmatrix},
      \qquad
      T\cdot\tau=\tau+1,
    \]
    and, in particular, $\tP_\infty\simeq \Z$.
    Moreover, for $d\gg 1$ the $\Gamma$-action on $\mathcal U_d$ is effectively reduced to the cusp stabilizer: if $\gamma\in\Gamma$ satisfies $\gamma(\mathcal U_d)\cap \mathcal U_d\neq\varnothing$, then $\gamma\in\tP_\infty$.
    Equivalently, the quotient of a sufficiently deep cusp neighborhood is modeled by the parabolic quotient
    \[
      \Gamma\backslash \mathcal U_d \;\cong\; \tP_\infty\backslash \mathcal U_d,
    \]
    so $\Gamma\backslash \mathcal U_d$ is a standard cusp neighborhood in $A_1$. 
    
    Now, for $g=1$, Equation~\eqref{eq:wp.metric} recovers the Weil--Petersson metric on $A_1$ in terms of the standard hyperbolic (Poincar\'e) metric:
    \begin{equation}\label{eq:wp-metric-g1}
      g_{\scaleto{W\!P}{3.5pt}}
      \;=\;
      \frac{1}{2}\frac{|d\tau|^2}{\im(\tau)^2}
      \;=\;
      \frac{1}{2}\frac{dx^2+dy^2}{y^2},
    \end{equation}
    with associated K\"ahler form
    \[
      \omega_{\scaleto{W\!P}{3.5pt}} \;=\; \frac{i}{4}\,\frac{d\tau\wedge d\bar\tau}{\im(\tau)^2}.
    \]
    In particular, at a point $\tau=x+iy$ we have
    \[
      \|\partial_x\|_{\scaleto{W\!P}{3.5pt}}=\frac{1}{2}\frac{1}{y},
      \qquad
      \|\partial_y\|_{\scaleto{W\!P}{3.5pt}}=\frac{1}{y}.
    \]
    Thus, as $y=\im(\tau)\to\infty$ (approaching the cusp), motion in the $x$--direction becomes arbitrarily short: a fixed displacement $\Delta x$ has WP-length $\sim |\Delta x|/y\to 0$. Furthermore, since the $\tP_\infty$-action identifies $x\sim x+n$ for $n\in\Z$, the $x$--direction compactifies to an $\mathbb{S}^1$; combined with the previous estimate, this shows that the $\mathbb{S}^1$ collapses in the limit. A more explicit geometric description of this behavior is obtained via the following coordinate change on $\tP_\infty\backslash\mathcal{U}_d$:
    \[
    r=\log \,y,\qquad \theta=x \:\:(\mathrm{mod}\,\Z).
    \]
    Then \eqref{eq:wp-metric-g1} becomes the standard cusp form
    \begin{equation}\label{eq:wp-cusp-form-g1}
      \g_{\scaleto{W\!P}{3.5pt}}
      \;=\;
      \frac{1}{2}\Big(dr^2 + e^{-2r}d\theta^2\Big).
    \end{equation}
    This makes the collapse explicit: along the end $r\to\infty$ the circle direction $\theta$ is multiplied by the shrinking factor $e^{-r}$, while the base direction $r$ has bounded size. In particular, the diameter of the fiber circles satisfies
    \[
      \mathrm{diam}_{\scaleto{W\!P}{3.5pt}}(\mathbb{S}^1_r)
      =
      \frac{1}{\sqrt{2}}\,e^{-r}
      \;\xrightarrow[r\to\infty]{}\;0.
    \]
    Thus, the asymptotic Weil--Petersson geometry of $A_1$ exhibits a collapsing circle fibration over the ray $r>\log \,d$.

    To give a coordinate-free description of this phenomenon, let us recall from Section~\ref{subsec:boun.comp.Sg} that, relative to the parabolic subgroup $P_\infty$ of the cusp, $\tP_\infty\backslash\mathcal{U}_d~\subset~A_1$ admits the structure of a fibration given by the map
    \begin{equation}
        \pi:\tP_\infty\backslash\cH \to \sym(1,\R)\cong\R^{+},
    \end{equation}
    whose fiber is a compact nilmanifold, here given by $\tN_\infty\backslash N_\infty \cong \mathbb{S}^{1}$. Therefore, the analysis above is simply the explicit coordinate description of the collapsing behavior of this fibration onto its base $\R^+$.
\end{example}

\begin{example}
\label{ex:wpabsurf}
We now turn to investigate the asymptotic Weil--Petersson geometry of the moduli space of principally polarized abelian surfaces. The main point is that this case is rich enough to exhibit the basic structural features of the asymptotics that we aim at capturing in the general setting.

We work on $A_2=\Gamma\backslash\fS_2$ and assume that $\fF_2(u)\subset\fS_2$ is a Siegel set adapted to the distinguished boundary components, namely $\cH$, and $*_\infty$. Given the Jacobi's decomposition $\tau=X +i Y = X + i (L\,D\,L^t)$
\begin{equation}
    L=\begin{pmatrix}
        1&0\\
        l_{21}&1
    \end{pmatrix},\qquad 
    D=\begin{pmatrix}
        d_1&0\\
        0&d_2
    \end{pmatrix},
\end{equation}
and the bounds imposed by the Siegel set $\fF_2(u)$ we have that a convenient set of local coordinates is\footnote{Here we are working with the convention that the coordinates of a point $\tau \in \fS_2$, are given by the embedding
\[\mathrm{Sym}_2(\C)\to \C^3, 
    \begin{pmatrix}
        z_1&z_2  \\
        z_2&z_3 
    \end{pmatrix} \mapsto 
    \begin{pmatrix}
        z_1\\
        z_2\\
        z_3 
    \end{pmatrix}. 
\]}
\begin{equation}
    (\theta_1,\dots,\theta_4)=(x_1,x_2,x_3,l_{21}),\qquad (r_1,r_2)=(\log\,d_1,\log\,d_2).
\end{equation}
With respect to this coordinates we have that the Weil--Petersson metric is given by 
\begin{equation*}
    \begin{aligned}
        \g_{\scaleto{W\!P}{3.5pt}}
        =\frac12\Bigg[
        \,dr_1^2+dr_2^2+2e^{r_1-r_2}\,d\theta_4^{\,2}
        +\frac{\bigl(1+e^{r_1-r_2}\theta_4^{\,2}\bigr)^2}{e^{2r_1}}\,d\theta_1^{\,2}
        +2\,\frac{\bigl(1+2e^{r_1-r_2}\theta_4^{\,2}\bigr)}{e^{r_1}e^{r_2}}\,d\theta_2^{\,2}
        \\[4pt]
        \quad
        +\frac{1}{e^{2r_2}}\,d\theta_3^{\,2}
        -4\theta_4\,\frac{\bigl(1+e^{r_1-r_2}\theta_4^{\,2}\bigr)}{e^{r_1}e^{r_2}}\,d\theta_1\,d\theta_2
        +2\,\frac{\theta_4^{\,2}}{e^{2r_2}}\,d\theta_1\,d\theta_3
        -4\,\frac{\theta_4}{e^{2r_2}}\,d\theta_2\,d\theta_3
        \Bigg].
    \end{aligned}
\end{equation*}

We restrict our analysis to the regime corresponding to the boundary component $\cH$. In terms of our choice of Siegel set data, this means that a sequence $\tau_n\in \fF_2(u)\subset\fS_2$ converges to a point $\tau_{\infty}\in \fF_1(u)\subset\cH$ if and only if $(\tau_n)_1\to \tau_\infty$ and $d_2(\tau_n)\to \infty$. In particular, if we work on the closed metric balls $\overline B_{\WP}(\tau_n,R)$ with $R\gg 1$, then for $n$ sufficiently large we have $d_1\ll d_2$ uniformly on $\overline B_{\WP}(\tau_n,R)$. 

Under the above setup, the Weil--Petersson metric admits the following asymptotic expansion
\begin{equation}
    \label{eq:wp.asymp.surf}
    \g_{\WP}
    = \frac{1}{2}\Bigl( dr_1^2+ e^{-2r_1}d\theta_1^2 \Bigr)
      + \frac{1}{2}\,dr_2^2
      + O\!\bigl(e^{-r_2(\tau_n)}\bigr).
\end{equation}
Heuristically, \eqref{eq:wp.asymp.surf} should be interpreted as saying that, in the direction of the boundary component $\cH$, the
Weil--Petersson geometry decouples into a lower-dimensional hyperbolic factor (in the $(r_1,\theta_1)$-variables)
together with an additional rank-one symmetric space direction (the $r_2$-variable), up to an error that is
uniformly small on bounded $\g_{\WP}$-balls centered at $\tau_n$.

To make this precise, recall that relative to the parabolic subgroup $P_{\cH}$ there is a natural fibration
\begin{equation}
    \label{eq:parabolic.fibration}
    \pi: \tP_{\cH}\backslash\fS_2\longrightarrow \tL_{\cH}\backslash(\cH\times \R^{+}),
    \quad
    \tau\longmapsto\bigl(\tau_1,\ y_3-y_2y_1^{-1}\bigr).
\end{equation}
Endow the total space with the Weil--Petersson metric $g_{\WP}$, and the base with the product metric
\begin{equation*}
\frac{1}{2}\Bigl(g_{\scaleto{\cH}{5pt}}+g_{\scaleto{\R^{+}}{5pt}}\Bigr),
\end{equation*}
where $\tfrac{1}{2}g_{\scaleto{\cH}{5pt}}$ is the Weil--Petersson metric on the lower-dimensional quotient
$\tL_{\cH,h}\backslash\cH$ and $\tfrac{1}{2}g_{\scaleto{\R^{+}}{5pt}}$ is the canonical (suitably normalized) symmetric-space
metric on $\R^{+}$. A direct computation then shows that \eqref{eq:wp.asymp.surf} can be rephrased as 
\begin{equation}
    \label{eq:wp.pullback.base}
    \g_{\WP}
    = \frac{1}{2}\,\pi^*\Bigl(g_{\scaleto{\cH}{5pt}}+g_{\scaleto{\R^{+}}{5pt}}\Bigr)
      + O\!\Bigl(e^{-r_2(\tau_n)}\Bigr)
    \qquad
    \text{on }\ \overline B_{\WP}(\tau_n,R).
\end{equation}
Here the error term is understood in the sense that, once $R>0$ is fixed, the coefficients of the difference tensor
\[
\g_{\WP}-\frac{1}{2}\,\pi^*\Bigl(g_{\scaleto{\cH}{5pt}}+g_{\scaleto{\R^{+}}{5pt}}\Bigr)
\]
are bounded by $C\,e^{-r_2(\tau_n)}$ throughout the closed ball $\overline B_{\WP}(\tau_n,R)$, for some constant $C$ that depends on the Siegel set data, $R$, and $\tau_\infty$. Therefore, \eqref{eq:wp.pullback.base} should be viewed as a \emph{uniform, local} approximation statement:
it compares the two metrics on a bounded $g_{\WP}$-neighborhood of the moving basepoint $\tau_n$.

One could now strengthen \eqref{eq:wp.pullback.base} by estimating the $C^0$-norm of the difference tensor on
$\overline B_{\WP}(\tau_n,R)$ with respect to either of the two metrics, thereby obtaining a genuine
uniform $C^0$-closeness result. However, attention must be paid to the fact that for a fixed $R$, and $n$ the two tensors can not be made arbitrarily close. Instead we need to work for a fixed $R$ and let $\tau_n$ vary as $n\to \infty$. In particular, the space on which the comparison is made changes every time.

We do not pursue this analysis further. Instead, in the next section we will give a characterization of the asymptotics of $\g_{\WP}$ in terms of pointed Gromow-Hausdorff limits, which we
find to be the most natural framework for describing the large-scale geometry in the non-compact setting relevant
to our applications.
\end{example}

Finally, in light of the essential role played by the moduli space of real flat tori and its canonical invariant metric in our analysis of the asymptotic behaviour of the Weil--Petersson metric, we introduce the following
\begin{definition}
    \label{def:trop.wp}
    We define the tropical Weil--Petersson metric $\g_{\trWP}$ on the moduli space of real flat tori of dimension $r$ to be
    \begin{equation}
        \label{eq:trop.wp}
        \g_{\trWP}=\frac{1}{2}tr(P^{-1}\,dP\,P^{-1}\,dP).
    \end{equation}
\end{definition}

Besides $\g_{\trWP}$ having the expected normalization appearing in the Weil--Petersson metric asymtotic in Examples \ref{ex:wp.ellip}, and \ref{ex:wpabsurf} one can give a clear modular interpretation to it as follows. Let $T^r=\R^r/\Z^r$ be the standard smooth torus, and let $P\in\sym(r,\R)$. Then $P$ defines a translation--invariant (hence flat) Riemannian metric $g_P$ on $T^r$ by
\[
g_P=\sum_{i,j=1}^r p_{ij}\,dx_i\,dx_j,\qquad P=(p_{ij}).
\]
Therefore, interpreting a tangent vector $V_P\in T_P\sym(r,\R)$ as the germ of a one parameter family of flat tori and recalling that we have $T_P\sym(r,\R)\cong\mathrm{Sym}(r,\R)$ given by the map 
\[
\Psi_P: T_P\sym(r,\R)\to \mathrm{Sym}(r,\R), V\to P^{-1/2}\,V\,P^{-1/2}
\]
Then, we can interpret $\g_{\trWP}$ as the natural metric measuring the variation of the flat Riemannian metric on the standard $r$-torus.

\subsection{Pointed Gromov-Hasudorff limits of \texorpdfstring{$(A_g,\,\g_{\WP},\tau_n)$}.}
In this section we give a proof of Theorem \ref{thm:main.theorem}. First, we make some notions precise:
\begin{definition}\label{def:non.deg.direc.Ag}
    Set $0\leq g'<g$, and let $\tau_n$ be a sequence in $A_g$, such that $\tau_n\to\tau'$ for some $\tau'\in A_{g'}$. Let $(A_g,\g_{\WP},\tau_n)$ be a sequence of pointed spaces. Then, by the asymptotics of the Weil--Petersson metric in the direction given by the sequence $\tau_n$, we mean the pointed Gromov--Hausdorff limit 
    \[
    (A_g,\g_{\WP},\tau_n)\xrightarrow[]{pG\!H}(Z,d_Z,z).
    \]
    Furthermore, we say that such a sequence defines a non-degenerate direction at infinity if and only if with respect to the Jacobi's decomposition $\tau_n=X_n + i \,L_n\,D_nL_n^t$ the following condition holds
    \begin{equation}
        \lim_{n\to \infty} \frac{d_i(\tau_n)}{d_j(\tau_n)}=k_{ij}\neq 0,\quad \text{for all}\,\, g'<i<j\leq g. 
    \end{equation}
\end{definition}

\begin{remark}
The reason it is natural to describe the asymptotic behavior of the Weil--Petersson metric in terms of pointed Gromov--Hausdorff limits is already suggested by our analysis in Examples~\ref{ex:wp.ellip}, and \ref{ex:wpabsurf}. Since $A_g$ is non-compact, there is in general no background-independent notion of global convergence for the tensor $\g_{\WP}$ on all of $A_g$. Instead, one compares the geometry on larger and larger pointed balls: for a degenerating sequence $\tau_n$, one studies the metric spaces $\bigl(\overline B_{A_g}(\tau_n,R),\g_{\WP}\bigr)$
for fixed $R>0$ and lets $n\to\infty$. Pointed Gromov--Hausdorff convergence is precisely designed to capture this type of local-to-global asymptotic behavior.
\end{remark}

Before we proceed with our general argument we motivate the discussion by giving a detailed computation for the moduli of elliptic curves.

\begin{example}[Gromov--Hausdorff collapse of the moduli of elliptic curves]\hfill \break
    \newline With the notation and set up of Example \ref{ex:wp.ellip} we have the following
    \begin{proposition}
        Let $\tau_n \in \tP_\infty\backslash\mathcal{U}_d \subset A_1$ be a sequence such that $\im(\tau_n)\to \infty$. Then,
        the fibration
        \begin{equation}
            \pi:\tP_\infty\backslash\cH\to \R^+,
        \end{equation}
        collapses to its base, as $\im(\tau_n)\to \infty$, in the pointed Gromov-Hausdroff sense. In particular we have that 
        \begin{equation}
        \label{eq:pGH.conv.A1}
            (A_1, \g_{\WP}, \tau_n) \xrightarrow[]{pG\!M} (\R^{+}_t, dt^2/2t^2,1),
        \end{equation}
        in the pointed Gromov--Hausdorff sense. 
    \end{proposition}
    \begin{proof}
    Although the statement is classical and surely familiar to experts, we provide a proof for completeness and for the convenience of readers who are not accustomed to Gromov--Hausdorff convergence.
    
    As shown in Example \ref{ex:wp.ellip}, it is more convenient to work with the coordinates
    $r=\log y$ and $\theta=x \,(\text{mod }\Z)$ on $\tP_\infty\backslash\cH$, and $s=\log t$ on $\R^{+}$.
    In these coordinates, the Weil--Petersson metric and the base metric take the form
    \[
    \g_{\WP}=\frac{1}{2}\Bigl( dr^2+ e^{-2r}d\theta^2\Bigr), \quad g_{\scriptscriptstyle{\R}}=\frac{1}{2}ds^2.
    \]
    To show that the fibration collapses to its base, it suffices to prove that for any $R>0$ and any $\epsilon>0$
    there exists an $N$ such that, for every $n>N$, the map
    \[
    \pi:\overline B_{\scaleto{\tP_\infty\backslash\cH}{6pt}}(\tau_n,R)\to \overline B_{\R}(s_n,R)
    \]
    is an $\epsilon$-pointed Gromov--Hausdorff approximation. To see this, note first that $\pi$ is a Riemannian
    submersion; in particular, it is $1$--Lipschitz. Moreover, the fiber diameter over the closed ball
    $\overline B_\R(s_n,R)$ admits the uniform bound
    \[
    \diam_{\WP}\bigl(\pi^{-1}(s)\bigr)\leq C_R \,e^{-s_n}, \quad \text{for any } s\in \overline B_{\R}(s_n,R).
    \]
    It follows from this observations that for any $\tau_p,\tau_q \in \overline B_{\scaleto{\tP_\infty\backslash\cH}{6pt}}(\tau_n,R)$ we have
    \[
    \|d_{\WP}(\tau_p,\tau_q)- d_{\scriptscriptstyle{\R}}(s_p,s_q)\|\leq C_R\,e^{-s_n}.
    \]
    Therefore, there exists a constant $C'_R>0$ such that
    \[
    d_{GH}\Bigl(\overline B_{\scaleto{\tP_\infty\backslash\cH}{6pt}}(\tau_n,R),\overline B_{\R}(s_n,R)\Bigr)\leq C'_R\,e^{-s_n}.
    \]
    Consequently, choosing $N>0$ so that $C'_R\,e^{-s_N}\leq \epsilon$ yields the desired pointed
    Gromov--Hausdorff approximation.
    
    Finally, to obtain the pointed Gromov--Hausdorff convergence in \eqref{eq:pGH.conv.A1}, choose $s_0=s_n$. Then the translation
    $s\mapsto s'=s-s_0$ is an isometry of $(\R,g_{\scriptscriptstyle{\R}})$ sending the basepoint $s_n$ to $0$.
    In the $t$-coordinate, this corresponds to the isometry $t\mapsto t'=e^{-s_n}t$, which sends $t_n=e^{s_n}$ to $1$.
    Thus, composing $\pi$ with this translation yields the desired pointed Gromov--Hausdorff approximation.
    \end{proof}
\end{example}

We now proceed with the proof of our main theorem. First let us summarize the set up and notation for our analysis. Throughout, we assume that we are working relative to the (maximal) chain of rational boundary components
\[
    \fS_{g-1} \succ \dots \succ \fS_1 \succ \fS_0,
\]
given by the complete isotropic flag 
\[
    \langle e_g\rangle\subset\langle e_{g-1},e_g\rangle\subset\cdots\subset\langle e_1,\cdots, e_g\rangle,
\]
For each $0\leq g'<g$, let $P_{g'}$ denote the corresponding parabolic subgroup associated with $\fS_{g'}$. We fix the Levi decomposition of $P_{g'}\cong N_{g'}L_{g'}$ associated with the isotropic-co-isotropic flag $I_{g'}\subset I_{g'}^\perp$ such that
\[
\fS_{g}\cong N_{g'}\fS_{g'}\sym(g'',\R),
\qquad
g''=g-g',
\]
and set 
\[
\tP_{g'}:=\Gamma\cap P_{g'},\quad \tN_{g'}:=\Gamma\cap N_{g'},\quad \tL_{g'}:=\Gamma\cap L_{g'}.
\]

We fix a suitable $u>0$ and work with the Siegel sets $\fF_{g'}(u)$, for $0\leq g'\leq g$. Furthermore, we assume they are adapted to the boundary components so that for any sequence $\tau_n\in\fF_g(u)\subset\fS_g$ we have
\begin{equation}\label{eq:sieg.set.direc}
    \tau_n\to \tau'_\infty \in \fF_{g'},\quad \text{if and only if}\quad  d_{g'+1}(\tau_n)\to \infty, \,\text{ and }\, \tau'_n\to \tau'_\infty.
\end{equation}
Here, for any $\tau =X+iY \in \fS_g$ we use the block decomposition   
\begin{equation}
    \tau=\begin{pmatrix}
        \tau'&\tau'''\\
        {}^t\tau'''&\tau''
    \end{pmatrix},
\end{equation}
relative to the boundary component $\fS_{g'}$, that we extend in the obvious way to its real and imaginary part. 

With this notation in place, we turn to a preliminary analysis that will allow for a cleaner presentation of the main ideas in the proof. The key observation, already visible in Examples~\ref{ex:wp.ellip} and~\ref{ex:wpabsurf}, is that the geometry of $A_g$ near a given boundary component is governed by the fibration induced by the horospherical decomposition relative to the associated parabolic subgroup. Crucially, the base of this fibration is the product 
$A_{g'}\times\cT{g''}$---precisely the space appearing in the 
conclusion of Theorem~\ref{thm:main.theorem}. Thus, showing that this fibration collapses to its 
base in the pointed Gromov--Hausdorff sense will establish the theorem. 

Concretely, for each $0 \leq g' < g$, we restrict our analysis to the region
\[
\mathcal{U}_{g'\!,d} := \{\tau \in \fF_g(u) \mid d_{g'}(\tau) > d\},
\]
with $d \gg 1$ sufficiently large. Then, by Proposition~\ref{prop:inv.neigh}, 
for $d$ large enough the quotient $\tP_{g'} \backslash \overline{\mathcal{U}}_{g'\!,d}$ embeds into 
$\Gamma \backslash \fS_g$, so we may work on $\tP_{g'} \backslash \fS_g$.
Finally, using the horospherical decomposition of $\fS_g$ with respect to $P_{g'}$, 
we obtain
\begin{equation}\label{eq:fibration-general}
\pi_{g'} : \tP_{g'} \backslash \fS_g \to \tL_{g'} \backslash 
\big(\fS_{g'} \times \sym(g'', \mathbb{R})\big),
\end{equation}
where $\pi_{g'}(\tau) = (\tau',t)$ with $t = Y'' - {}^tY'''(Y')^{-1}Y'''$.

We can express this more suggestively by recalling that 
\[
\tL_{g'} = \tL_{g'h} \times \tL_{g'l} = \Sp(2g',\Z) \times \mathrm{GL}(g'',\Z).
\]
Thus, Equation~\eqref{eq:fibration-general} becomes
\begin{equation}\label{eq:fibration-modular}
\pi_{g'} : \tP_{g'} \backslash \fS_g \to A_{g'} \times 
\cT{g''}.
\end{equation}
This identification shows that the base of the fibration is precisely the product of 
the corresponding boundary strata in the volume-scale and diameter-scale Gromov--Hausdorff 
compactifications of $A_g$ (cf. Remark~\ref{rmk:sbb.trop.strata} and the diagram on 
page $10$).

From our preliminary analysis, we can restate the main theorem as follows.

\begin{theorem}
\label{thm:main.second.version}
    Let $\tau_n$ be a sequence in $\fF_g(u)\subset\fS_g$ such that
    $\tau_n\to \tau'_\infty\in\fF_{g'}(u)$ for some $0\leq g' <g$.
    This naturally defines a sequence in $\tP_{g'}\backslash \fS_g$, which we keep denoting by $\tau_n$.
    Then the fibration
    \[
        \pi_{g'}: \tP_{g'}\backslash\fS_g \longrightarrow A_{g'}\times\cT{g''},
    \]
    collapses to its base in the pointed Gromov--Hausdorff sense. More precisely, for each fixed $R>0$,
    \begin{equation}
        \label{eq:GH.dis.balls.R}
        d_{GH}\Big(B(\tau_n,R),\, B(\pi_{g'}(\tau_n),R)\Big)
        \le \frac{C_R}{\sqrt{d_{g'+1}(\tau_n)}} \xrightarrow[n\to\infty]{} 0,
    \end{equation}
    for some constant $C_R$ depending only on the Siegel set data, on $R$, and on the limit point $\tau'_\infty$.
    Furthermore, for $n$ sufficiently large, we may (and will) regard $\tau_n$ as a sequence in
    \[
    \tP_{g'}\backslash\overline{\mathcal{U}}_{g'\!,d}\simeq
    \Gamma\backslash\overline{\mathcal{U}}_{g'\!,d}\subset A_g.
    \]
    Assuming moreover that $\tau_n$ satisfies
    \begin{equation}
    \label{eq:lim.didj}
        \lim_{n\to \infty}\frac{d_i(\tau_n)}{d_j(\tau_n)}=k_{i,j}\neq 0
        \quad \text{for all }\,g'+1\leq i<j\leq g,
    \end{equation}
    we obtain
    \begin{equation}\label{eq:pGH.conve.Ag}
        (A_g,\,\g_{\WP},\,\tau_n)\xrightarrow[]{\,\,pG\!H\,\,}
        (A_{g'}\times\cT{g''},\,\g_{\WP} + \g_{\trWP}, \,(\tau'_\infty,t_\infty)).
    \end{equation}
\end{theorem}

Let us first give a brief overview of the idea of the proof. Our goal is to show three properties of this fibration:
\begin{enumerate}[label=(\roman*)]
\item The projection map $\pi_{g'}$ is length-nonincreasing;
\item Given a metric ball of arbitrary radius on the base, we have a uniform diameter bound on the the fibers depending only on the center and radius of the ball;
\item in the direction of the boundary, the fiber diameters vanish.
\end{enumerate}
These three properties immediately imply that the fibration collapses to its base in 
the pointed Gromov--Hausdorff sense.

\begin{lemma}
    \label{lem:riem.subm.gen}
    Fix $0\leq g'< g$, and equip the total space and base of the fibration 
    \[
    \pi_{g'} : \tP_{g'} \backslash \fS_g \to A_{g'} \times \cT{g''},
    \]
    with the Riemannian structures given, respectively, by $\g_{\WP}$ and by the product metric $\g'_{\WP} + \g''_{\trWP}$. Then, the projection $\pi_{g'}$ is a Riemannian submersion and in particular it is $1$--Lipschitz. 
\end{lemma}

\begin{proof}
Since being a Riemannian submersion is local property we may work on the universal covers \footnote{Recall that, by a harmless abuse of notation, we denote by $\pi_{g'}$ both the projection on the universal covers and the induced projection on the quotients.}
\[
\pi_{g'}:\fS_g\longrightarrow \fS_{g'}\times \sym(g'',\R).
\]
Furthermore, since $P_{g'}$ acts transitively and by isometries and $\pi_{g'}$ is $P_{g'}$-equivariant, it suffices to verify the claim at the base point $iI_g\in\fS_g$. Consider the differential
\[
(\pi_{g'*})_{iI_g}:\,T_{iI_g}\fS_g \longrightarrow
T_{(iI_{g'},\,I_{g''})}\bigl(\fS_{g'}\times \sym(g'',\R)\bigr).
\]
Using the identification $T_{iI_g}\fS_g\simeq \mathrm{Sym}(g,\C)$, write a tangent vector as
\[
V=V_X+iV_Y,\qquad V_X,V_Y\in\mathrm{Sym}(g,\R),
\]
and decompose into blocks according to $g=g'+g''$:
\[
V_X=
\begin{pmatrix}
V_X' & V_X'''\\
{}^tV_X''' & V_X''
\end{pmatrix},
\qquad
V_Y=
\begin{pmatrix}
V_Y' & V_Y'''\\
{}^tV_Y''' & V_Y''
\end{pmatrix}.
\]

A direct computation yields
\[
(\pi_{g'*})_{iI_g}(V)=\bigl(V_X'+iV_Y',\,V_Y''\bigr),
\]
so $(\pi_{g'*})_{iI_g}$ is surjective and its kernel is
\[
\ker\,(\pi_{g'*})_{iI_g}
=
\{\,V\in T_{iI_g}\fS_g:\ V_X'=0,\ V_Y'=0,\ V_Y''=0\,\}.
\]
At $iI_g$ the Weil--Petersson metric gives
\[
\|V\|^2_{iI_g}=\frac{1}{2}\bigl(tr(V_X^2)+tr(V_Y^2)\bigr),
\]
and the block trace identities give
\[
tr(V_X^2)=tr\big((V_X')^2\big)+2tr\big(V_X'''\,{}^tV_X'''\big)+tr\big((V_X'')^2\big),
\]
and similarly for $V_Y$. It follows that the orthogonal complement of the kernel is precisely
\[
\ker\,(\pi_{g'*})_{iI_g}^\perp
=
\left\{
V\in T_{iI_g}\fS_g:\ 
V=
\begin{pmatrix}
V_X'+iV_Y' & 0\\
0 & iV_Y''
\end{pmatrix}
\right\}.
\]
For such $V$ one has
\[
\|V\|^2_{iI_g}
=
\frac{1}{2}\bigl(tr\big((V_X')^2\big)+tr\big((V_Y')^2\big)+tr\big((V_Y'')^2\big)\bigr),
\]
which is exactly the product metric on
$T_{(iI_{g'},\,I_{g''})}\bigl(\fS_{g'}\times\sym(g'',\R)\bigr)$.
Therefore $(\pi_{g'*})_{iI_g}$ restricts to an isometry
\[
\ker\bigl((\pi_{g'*})_{iI_g}\bigr)^\perp
\;\xrightarrow{\ \simeq\ }\;
T_{(iI_{g'},\,I_{g''})}\bigl(\fS_{g'}\times\sym(g'',\R)\bigr),
\]
and hence $\pi_{g'}$ is a Riemannian submersion.
\end{proof}

As a first step in obtaining a uniform bound on the fiber diameter, we derive in the 
following two lemmas an explicit bound on the norm of vertical tangent vectors in 
terms of the base point $(\tau', t)$.

\begin{lemma}
    \label{lem:vert.norm.exp}
    Let $\fF_{g}(u)\subset\fS_{g},\fF_{g'}(u)\subset\fS_{g'}$, and $\fF^{+}_{g''}(u)\subset\sym(g'',\R)$ be fixed Siegel sets and let $(\tau',t)\in \fF_g(u)\times\fF_{g''}(u)$ be a fixed point in the base of fibration suitably restricted to
    \[
    \pi_{g}:\fF_g(u)\longrightarrow\fF_{g'}(u)\times\fF^{+}_{g''}(u).
    \]
    let $V\in T_\tau \fF_g(u)$, be a vertical tangent vector at $\tau\in \pi^{-1}_{g'}(\tau',t)\cap \fF_g(u)$. Then, its norm satisfies the identity
    \begin{equation}
    \label{eq:normvert}
        \|V\|_{\tau}^2
        =
        \,tr\!\bigl(t^{-1}\,{}^tV_X'''(Y')^{-1}V_X'''\bigr)
        +
        \,tr\!\bigl(t^{-1}\,{}^tV_Y'''(Y')^{-1}V_Y'''\bigr)
        +
        \frac{1}{2}tr\!\bigl(t^{-1}\widetilde V_X''\,t^{-1}\widetilde V_X''\bigr),
    \end{equation}
    where $\widetilde V_X''$ is defined below in Equation \eqref{eq:v.tilde.pp}.
\end{lemma}
\begin{proof}
    On the fiber over $(\tau',t)$ we use the coordinates
    \[
        (X'''+iY''',\,X''),
        \qquad\text{so that}\qquad
        \tau''=X'' + i\bigl(t+{}^tY'''(Y')^{-1}Y'''\bigr).
    \]
    If $V\in T_{\tau}\fF_g(u)$ is vertical then
    \begin{equation}
        V'=0,\quad\text{and}\quad 
        V_Y'' ={}^tV_Y'''(Y')^{-1}Y'''+{}^tY'''(Y')^{-1}V_Y'''.
    \end{equation}
    In order to simplify the norm formula
    \begin{equation}
    \label{eq:normV}
        \|V\|_{\tau}^2=\frac{1}{2}\bigl(tr(Y^{-1}V_XY^{-1}V_X)+tr(Y^{-1}V_YY^{-1}V_Y)\bigr),
    \end{equation}
    we introduce the lower-triangular matrix
    \[
        P:=
        \begin{pmatrix}
            I_{g'} & 0\\
            -{}^tY'''(Y')^{-1} & I_{g''}
        \end{pmatrix}.
    \]
    A direct computation yields the identities
    \begin{equation}\label{eq:PYPt}
        P\,Y\,P^t=\mathrm{diag}(Y',t),
        \qquad
        Y^{-1}=P^t\,\mathrm{diag}((Y')^{-1},t^{-1})\,P.
    \end{equation}
    Moreover, the same change of variables gives
    \begin{equation}\label{eq:PVPt}
        P\,V_X\,P^t=
        \begin{pmatrix}
        0 & V_X'''\\
        {}^tV_X''' & \widetilde V_X''
        \end{pmatrix},
        \qquad
        P\,V_Y\,P^t=
        \begin{pmatrix}
        0 & V_Y'''\\
        {}^tV_Y''' & 0
        \end{pmatrix},
    \end{equation}
    where
    \begin{equation}\label{eq:v.tilde.pp}
        \widetilde V_X'':=
    V_X''-{}^tV_X'''(Y')^{-1}Y'''-{}^tY'''(Y')^{-1}V_X'''\in\mathrm{Sym}(g'',\R).
    \end{equation}
    Substituting \eqref{eq:PYPt} and \eqref{eq:PVPt} into \eqref{eq:normV} yields \eqref{eq:normvert}. 
\end{proof}

\begin{lemma}
\label{lem:vert.norm.est}
    Under the hypothesis of the previous lemma, there exists a constant $C>0$, depending only on the Siegel set data and $\tau'$, such that for every vertical tangent vector $V\in T_{\tau}\fF_g(u)$,
    \begin{equation}
        \label{eq:norm.est}
        \|V\|_{\tau}^2
        \le
        \frac{C}{\lambda_{\min}(t)}\Bigl(\|V_X'''\|_{\mathcal{F}}^2+\|V_Y'''\|_{\mathcal{F}}^2\Bigr)
        +
        \frac{C}{\lambda_{\min}(t)^2}\,\|V_X''\|_{\mathcal{F}}^2;
    \end{equation}
    here $\pi(\tau)=(\tau',t)$, $\lambda_{min}(t)$ denotes the smallest eigenvalue of $t$, and  by $\|\cdot\|_{\mathcal{F}}$ we denote the Frobenius norm.
\end{lemma}
\begin{proof}
    We have that $\|t^{-1}\|_{\mathrm{op}}=\lambda_{\min}(t)^{-1}$, where $\|\cdot\|_{\mathrm{op}}$ denotes the operator norm.
    Using the identity
    \[
        tr(t^{-1}\,{}^tB\,A^{-1}\,B)\le \|t^{-1}\|_{\mathrm{op}}\,\|A^{-1}\|_{\mathrm{op}}\,\|B\|_{\mathcal{F}}^2,
    \]
    the first two terms in \eqref{eq:normvert} satisfy
    \begin{equation}
        \label{eq:trace.norm.iden}
        \,tr\!\bigl(t^{-1}\,{}^tV_X'''(Y')^{-1}V_X'''\bigr)
        \le
        \frac{\|(Y')^{-1}\|_{\mathrm{op}}}{\lambda_{\min}(t)}\,\|V_X'''\|_{\mathcal{F}}^2,
    \end{equation}
    and similarly with $V_Y'''$. For the last term we bound
    \[
    tr(t^{-1}\widetilde V_X''\,t^{-1}\widetilde V_X'')
    \le
    \|t^{-1}\|_{\mathrm{op}}^2\,\|\widetilde V_X''\|_{\mathcal{F}}^2
    =
    \frac{1}{\lambda_{\min}(t)^2}\,\|\widetilde V_X''\|_{\mathcal{F}}^2.
    \]
    Finally, from the definition of $\widetilde V_X''$ we have
    \begin{equation}
        \label{eq:normVtilde}
        \|\widetilde V_X''\|_{\mathcal{F}}
        \le
        \|V_X''\|_{\mathcal{F}}
        +2\,\|Y'''\|_{\mathrm{op}}\,\|(Y')^{-1}\|_{\mathrm{op}}\,\|V_X'''\|_{\mathcal{F}}.
    \end{equation}
    Since the base point $(\tau',t)$ is fixed and $\|L\|_{op}$ is uniformly bounded on a fixed Siegel set we get that $\|Y'''\|_{\mathrm{op}}$ is uniformly bounded on the fiber. Hence, there are constants $C_1$, and $C_2$ such that $\|(Y')^{-1}\|_{\mathrm{op}}= C_1$ and $\|Y'''\|_{\mathrm{op}}\le C_2$, so \eqref{eq:normVtilde} gives
    \begin{equation}
    \label{eq:normVtilderefined}
        \|\widetilde V_X''\|_{\mathcal{F}}
        \le
        \|V_X''\|_{\mathcal{F}}+2\,C_1 C_2\,\|V_X'''\|_{\mathcal{F}}.
    \end{equation}
    Combining \eqref{eq:trace.norm.iden}--\eqref{eq:normVtilderefined} and absorbing constants yields the desired inequality.
\end{proof}

With these results at hand we can give the following point-wise estimate of the fiber diameter

\begin{proposition}\label{pro:fiber.diam.est}
    Let $\tau_n$ be a sequence in $\fF_g(u)\subset \fS_g$ such that $\tau_n\to \tau_\infty'\in \fF_{g'}(u)\subset \fS_{g'}$. This naturally defines a sequence in $\tP_{g'}\backslash\fS_g$, that we keep denoting by $\tau_n$. Let
    \[
        \pi_{g'}(\tau_n)=\big(\tau'_n\,,t_n\big)\in A_{g'}\times\cT{g''}.
    \]
    Then there exists a constant $C>0$ (depending only on the Siegel set and on $\tau'_\infty$) such that
    \begin{equation}\label{eq:diam-lambda}
        \diam\big(\pi_{g'}^{-1}\big(\tau'_n\,,t_n\big)\big) \le\ \frac{C}{\sqrt{\lambda_{\min}(t_n)}}.
    \end{equation}
    Furthermore, there exists $C'>0$ such that
    \begin{equation}\label{eq:diam-dgplus1}
        \diam\big(\pi_{g'}^{-1}\big(\tau'_n\,,t_n\big)\big)\leq \frac{C'}{\sqrt{d_{g'+1}(\tau_n)}}.
    \end{equation}
    In particular, as $\tau_n\to \tau'_\infty$ the diameter of the compact fiber vanishes.
\end{proposition}
\begin{proof}
Without loss of generality we work with the Siegel set description of the quotients, i.e.,
\[
    \pi_{g}:\fF_g(u)\longrightarrow\fF_{g'}(u)\times\fF^{+}_{g''}(u).
\]
Fix $(\tau',t)\in \fF_{g'}(u)\times\fF^{+}_{g''}(u)$ and let 
\[
    (X'''+iY''',\,X''),
    \qquad\text{so that}\qquad
    \tau''=X'' + i\bigl(t+{}^tY'''(Y')^{-1}Y'''\bigr).
\]
be the usual coordinates on the fiber over $(\tau',t)$. 

It is clear from the definiton of Siegel set that $\pi^{-1}_{g}(\tau',t)\cap \fF_g(u)$  gives a lattice fundamental domain of uniformly bounded size for the compact nilmanifold, possibly with finite orbifold structure, appearing as the fiber of the fibration\footnote{Recall that given a point $(\tau',t)$ in the base of the fibration $\pi_g: \tP_{g'}\backslash \fS_g\to A_{g'}\times\cT{g''}$ the fiber over it is homeomorpich to the (orbifold) compact nimanifold given by $\tL_{g',(\tau',t)}\backslash(\tN_{g'}\backslash N_{g'})$, where $\tL_{g',(\tau',t)}$ is the isotropy group of $(\tau',t)$}. Thus, there exists a constant $R>0$ (depending only on the chosen Siegel set) such that for any two points $\tau_p,\tau_q\in \pi_{g'}^{-1}(\tau',t)\cap\fF_g(u)$ the fiber-coordinate differences are bounded by $R$ in Frobenius norm, namely
\begin{equation}\label{eq:bounded-increments}
\|\Delta X'''\|_{\mathcal {F}}\le R,\qquad
\|\Delta Y'''\|_{\mathcal {F}}\le R,\qquad
\|\Delta X''\|_{\mathcal {F}}\le R.
\end{equation}

Given such $\tau_p,\tau_q$, connect them by a piecewise smooth path $\gamma(s)$ contained in the fiber, chosen so that $(\tau',t)$ remains constant and $X'''$ varies linearly on the first segment and, respectively, $Y'''$ and $X''$ on the second and third. Then $\dot\gamma(s)$ is vertical, and along each linear segment
\[
\|V_X'''\|_{\mathcal{F}}=\|\Delta X'''\|_{\mathcal{F}},\qquad
\|V_Y'''\|_{\mathcal{F}}=\|\Delta Y'''\|_{\mathcal{F}},\qquad
\|V_X''\|_{\mathcal{F}}=\|\Delta X''\|_{\mathcal{F}},
\]
and consequently
\[
\|V_X'''\|_{\mathcal{F}},\;\|V_Y'''\|_{\mathcal{F}},\;\|V_X''\|_{\mathcal{F}}\;\le R.
\]
Applying Lemma~\ref{lem:vert.norm.est} to $V=\dot\gamma(s)$ yields
\[
\|\dot\gamma(s)\|^2
\le
\frac{C_0}{\lambda_{\min}(t)}\Bigl(\|V_X'''\|_{\mathcal{F}}^2+\|V_Y'''\|_{\mathcal{F}}^2\Bigr)
+
\frac{C_0}{\lambda_{\min}(t)^2}\,\|V_X''\|_{\mathcal{F}}^2
\le
\frac{C_1}{\lambda_{\min}(t)},
\]
for a constant $C_1$ depending only on the Siegel set and $\tau'$. In particular, $\|\dot\gamma(s)\|\le C_2/\sqrt{\lambda_{\min}(t)}$. Integrating along $\gamma$ gives
\[
d(\tau_p,\tau_q)\le \mathrm{length}(\gamma)
\le \frac{C_{\tau'}}{\sqrt{\lambda_{\min}(t)}},
\]
uniformely over all such $\tau_p,\tau_q\in \pi_{g'}^{-1}(\tau',t)\cap \fF_g(u)$, with $C_{\tau'}$ depending only on the Siegel set and $\tau'$. Taking the supremum over $\tau_p,\tau_q$ yields 
\[
\diam\big(\pi_{g'}^{-1}\big(\tau'\,,t\big)\big) \le\ \frac{C_{\tau'}}{\sqrt{\lambda_{\min}(t)}}.
\]
Then, applying this to the sequence $(\tau'_n,t_n)$ and noticing that $\tau'_n\to \tau_\infty$ we get \eqref{eq:diam-lambda}.

Finally, \eqref{eq:diam-dgplus1} follows from the fact that, given a point $(\tau',t)$ and any $\tau\in \pi^{-1}_{g'}(\tau',t)$, the Siegel set inequalities $d_{i}~\leq~ u \,d_{i+1}$ imply the existence of $c>0$ such that
\begin{equation}\label{eq:lambdamin-lower}
    \lambda_{\min}(t)\ \ge\ c\, d_{g'+1}(\tau).
\end{equation}
To see this, write the Jacobi's decomposition $Y=LDL^t$ in block form as
\begin{equation}
       \begin{pmatrix}
            L'&0\\
            L'''&L''
        \end{pmatrix}
        \begin{pmatrix}
            D'&0\\
            0&D''
        \end{pmatrix}
        \begin{pmatrix}
            ^tL'&^tL'''\\
            0&^tL''
        \end{pmatrix}.
\end{equation}
This gives $t=L''\,D''\,{}^tL''$ with $D''=\mathrm{diag}(d_{g'+1},\cdots,d_{g})$. Let $x \in \R^{g''}$. Then
\[
{}^tx\, t\, x = {}^t\bigl({}^tL'' x\bigr)\,D''\,\bigl({}^tL'' x\bigr) \ge d_{min}\,\|{}^tL'' x\|^2.
\]
Using $\|{}^tL'' x\|\ge \frac{\|x\|}{\|(L'')^{-1}\|_{\mathrm{op}}}$ and taking the infimum over $\|x\|=1$ yields
\[
\lambda_{min}(t) \ge \frac{d_{min}}{\|(L'')^{-1}\|_{\mathrm{op}}^2},
\]
where $d_{min}=\mathrm{min}\{d_i\,|\, g'+1\leq i \leq g \}$. Thus \eqref{eq:lambdamin-lower} follows from $\|(L'')\|_{\mathrm{op}}$ being uniformly bounded on the Siegel set (and consequently $\|(L'')^{-1}\|_{\mathrm{op}}$), together with $d_{min} \geq d_{g'+1}/\mathrm{max}\{1,u^{g''-1}\}$.
\end{proof}

Finally we proceed with the proof of the main theorem

\begin{proof}[Proof of Theorem \ref{thm:main.second.version}]
    First, observe that for every $n$ and every $R>0$,
    \[
        \pi_{g'}\bigl(B(\tau_n,R)\bigr)=B\bigl(\pi_{g'}(\tau_n),R\bigr).
    \]
    Indeed, by Lemma \ref{lem:riem.subm.gen} the map $\pi_{g'}$ is $1$-Lipschitz, hence for any $\tau\in B(\tau_n,R)$ we have
    \[
        d\bigl(\pi_{g'}(\tau),\pi_{g'}(\tau_n)\bigr)\le d(\tau,\tau_n)\le R,
    \]
    which gives the inclusion $\pi_{g'}(B(\tau_n,R))\subset B(\pi_{g'}(\tau_n),R)$. For the reverse inclusion, fix $(\tau',t)\in B(\pi_{g'}(\tau_n),R)$ and let $\beta$ be a minimizing geodesic in the base joining $\pi_{g'}(\tau_n)$ to $(\tau',t)$. Since $\pi_{g'}$ is a Riemannian submersion, $\beta$ admits a horizontal lift $\tilde\beta$ starting at $\tau_n$, and the lift preserves length. In particular,
    \[
        \mathrm{leng}(\tilde\beta)=\mathrm{leng}(\beta) = d\bigl(\pi_{g'}(\tau_n),(\tau',t)\bigr)\le R,
    \]
    so $\tilde\beta(1)\in B(\tau_n,R)$ and $\pi_{g'}(\tilde\beta(1))=(\tau',t)$. This proves $B(\pi_{g'}(\tau_n),R)\subset \pi_{g'}(B(\tau_n,R))$, hence the desired equality.

    Next, we quantify the distortion of distances in terms of the fiber diameter. Define
    \[
        \delta(\tau_n;R):=\sup\Bigl\{\operatorname{diam}\bigl(\pi^{-1}_{g'}(\tau',t)\bigr)\ \Big|\ d\bigl(\pi_{g'}(\tau_n),(\tau',t)\bigr)\le R\Bigr\}.
    \]
    Then for all $\tau_p,\tau_q\in B(\tau_n,R)$, the usual argument using the $1$-Lipschitz property of $\pi_{g'}$ and a path connecting $\tau_p$ and $\tau_q$ consisting of an horizontal lift of a path in the base, concatenated with a vertical path inside the fiber over $\pi_{g'}(\tau_q)$, gives
    \[
        0\le d(\tau_p,\tau_q)-d\bigl(\pi_{g'}(\tau_p),\pi_{g'}(\tau_q)\bigr)\le \delta(\tau_n;R).
    \]
    Consequently,
    \[
        d_{GH}\Bigl(B(\tau_n,R),\, B(\pi_{g'}(\tau_n),R)\Bigr)\le \frac12\,\delta(\tau_n;R).
    \]

    It remains to estimate $\delta(\tau_n;R)$, i.e. to control fiber diameters uniformly over the base ball $B(\pi_{g'}(\tau_n),R)$. Fix $(\tau',t)\in B(\pi_{g'}(\tau_n),R)$. We claim that
    \begin{equation}
    \label{eq:lmin-ball}
    e^{-\sqrt{2}R}\,\lambda_{\min}(t_n)\leq \lambda_{\min}(t)\leq e^{\sqrt{2}R}\,\lambda_{\min}(t_n).
    \end{equation}
    In particular, there exists a constant $C$ (depending only on the chosen Siegel set data and on $\tau'_\infty$) such that
    \begin{equation}
        \label{eq:delta.fiber.unif}
        \delta(\tau_n;R)\leq \frac{C\,e^{R/\sqrt{2}}}{\sqrt{\lambda_{\min}(t_n)}}.
    \end{equation}
    Granting \eqref{eq:lmin-ball}, the estimate \eqref{eq:delta.fiber.unif} follows immediately from Proposition \ref{pro:fiber.diam.est}, Equation \eqref{eq:diam-dgplus1}, together with the definition of $\delta(\tau_n;R)$. Moreover, combining this with \eqref{eq:lambdamin-lower} and the fact that $\tau_n\to\tau'_\infty$ forces $d_{g'+1}(\tau_n)\to\infty$, one obtains \eqref{eq:GH.dis.balls.R}.

    We now prove \eqref{eq:lmin-ball}. Recall the standard distance formula on $\sym(g'',\mathbb R)$:
    \[
        d(t_n,t)^2=\frac12\sum_i \log^2(\mu_i),
    \]
    where $\mu_1\le \cdots\le \mu_{g''}$ are the eigenvalues of
    \[
        A:=t_n^{-1/2}\,t\,t_n^{-1/2}.
    \]
    Since $(\tau',t)\in B(\pi_{g'}(\tau_n),R)$, we have $d(t_n,t)\le R$, and hence
    \[
    \frac{1}{\sqrt2}\,|\log \mu_i|\le R\qquad\text{for all }i,
    \]
    so $e^{-\sqrt2 R}\le \mu_i\le e^{\sqrt2 R}$. In particular $\mu_{\min}(A)\ge e^{-\sqrt2 R}$ and $\mu_{\max}(A)\le e^{\sqrt2 R}$. Using the inequalities
    \[
    \lambda_{\min}(t)\ \ge\ \lambda_{\min}(t_n)\,\mu_{\min}(A),
    \qquad\text{and}\qquad
    \lambda_{\min}(t)\ \le\ \lambda_{\min}(t_n)\,\mu_{\max}(A),
    \]
    we obtain exactly \eqref{eq:lmin-ball}.

    Finally, observe that the condition in Equation \eqref{eq:lim.didj} can be precisely rephrased as saying that there exist a point $\overline t_\infty \in \tL_{g'l}\backslash\sym(g'',\R)/\R_{>0}$ such that the corresponding sequence $\overline t_n\to \overline t_\infty$. To see this, just notice that 
    \begin{equation}
        \begin{aligned}
            \overline{t}_\infty=\lim_{n\to \infty}\frac{1}{d_{g}(\tau_n)}L''_nD''_n{}^tL''_n=&\lim_{n\to \infty}\mathrm{diag}\Biggl(\frac{d_{g'+1}(\tau_n)}{d_{g}(\tau_n)},\cdots,\frac{d_{g-1}(\tau_n)}{d_{g}(\tau_n)},1 \Biggr)\\[5pt]
            =&\:\mathrm{diag}\bigl(k_{g'+1,g},\cdots,k_{g-1,g},1).
        \end{aligned}
    \end{equation}
    Thus, the pointed Gromov--Hausdorff convergence in Equation \eqref{eq:pGH.conve.Ag} follows by remembering that for any $R>0$ and for $n$ sufficiently large $B(\tau_n,R)\subset A_g$ can be assumed to be in the region where $\Gamma\sim \tP_{g'}$ and consequently we can use Equation \eqref{eq:GH.dis.balls.R} together. This finishes the proof 
\end{proof}

\subsection{Final remarks}

The above results establish an explicit connection between the Weil--Petersson metric asymptotics and the geometric compactifications of $A_g$ in terms of Gromov--Hausdorff limits of the (canonical) flat K\"ahler metric along degenerating sequences of abelian varieties. Furthermore, they show that at infinity the moduli space collapses to a lower-dimensional space and in particular such space does not longer have the structure of a complex algebraic variety. This result, although enigmatic, is consistent with the expectation that any \emph{geometric compactification} of the moduli space of polarized Calabi--Yau manifolds must account for the fact that, due to collapsing phenomena, the limiting objects are in general not algebraic varieties. 

However, let us emphasize that the limiting Weil--Petersson geometry of $A_g$ is not at all detached from an algebro-geometric interpretation. First, the holomorphic factor in the limit, namely $(A_{g'},\g'_{\WP})$, can be naturally identified with a boundary stratum of the Satake--Baily--Borel compactification $\bar A_g^{\sbb}$ of $A_g$. Which by construction is the minimal projective compactification of $A_g$. On the other hand, the non-holomorphic factor $(\cT{g''},\g''_{\trWP})$, that we will refer to as the tropical factor following Odaka, can be naturally identified with a stratum of the dual intersection complex of the boundary divisor of any toroidal compactification of $A_g$ (This in the sense explained by \cite{Odaka2019}). Moreover, from Theorem \ref{thm:Oda.holo.family} (Theorem $3.1$ in \cite{Odaka2019}) one can see that the holomorphic factor $A_{g'}$ and the tropical factor $\cT{g''}$ can be thought as parameterizing respectively the abelian and the compact torus part in a semi-abelian degeneration. 

The proof of Theorem~\ref{thm:main.theorem} relies on the fact that $A_g$ admits an explicit realization as a locally symmetric space (although the final theorem gives a modular interpretation which is independent from that). In particular, our approach is not specific to abelian varieties and it applies to moduli spaces whose global structure is governed by an arithmetic quotient of a symmetric space, such as the moduli spaces of polarized K3 surfaces or general polarized hyperähler manifolds. Let us briefly outline here this perspective and describe on-going work toward these extensions.

It is well known that the moduli space $\CMcal{F}_{2d}$ of (polarized) K3 surfaces of degree $2d$ can be realized as a locally symmetric space 
\begin{equation}
    \CMcal{F}_{d}=\Gamma_{2d}\backslash\CMcal{D}_{2d},
\end{equation}
where $\CMcal{D}_{2d}$ is a Hermitian symmetric domain of type $\mathrm{IV}$ and $\Gamma_{2d}$ is an arithmetic group acting on $\CMcal{F}_{2d}$. It is not difficult to see that the constructions introduced in Section \ref{subsec:boun.comp.Sg} generalized accordingly to $\CMcal{D}_{2d}$. In particular, we have that there are only two types of boundary components $E\subset \overline{\CMcal{D}}_{2d}$ and associated to each (rational) maximal parabolic subgroup $P_E$ we have the corresponding horospherical decomposition of $\CMcal{D}_{2d}$. In on-going work we show that our result in Theorem \ref{thm:main.second.version} extend to the case of $\CMcal{F}_{2d}$ and, similarly, the picture connects naturally with the works of Odaka and Oshima concerning diameter scale limits, as well as the differential geometric meaning for the Satake-Baily-Borel in terms of volume limits as suggested in \cite{spotti25} Conjectural picture $2$.

Finally, let us conclude with two more speculative comments. Firstly, note that along the collapsing process of the Weil--Petersson metric, important metric information is lost. For instance, the limiting asympotitic alone is not sufficient to reconstruct the full Weil--Petersson metric at infinity. For example, the baby case of  moduli of flat metrics with cone angle on the projective line (see, e.g., \cite{spotti25} Section $2$) suggests that information of \textit{higher scales bubbling} may be important as well: in such case the Weil--Petersson cusps degenerate to a real line, but the full metric can be recover as a Calabi's ansatz from a collapsed complex link, see \cite{BS19} Proposition $5$, which is naturally identified with moduli of non-collapsing minimal cylindrical bubbles. Secondly, note that a very recent breakthrough \cite{BFMT2025} has constructed Satake-Baily-Borel type compactification for general polarized Calabi-Yau varieties. In particular, it is natural to expect that similar holomorphic plus tropical splitting phenomena as we have described can occur in much greater generality for the Weil--Petersson asymptotic: the holomorphic part should relate to the Weil--Petersson metric generated by moduli of log canonical centers in the above generalized Satake-Baily-Borel type compactification, while a more mysterious tropical part should relate to aspects of collapsing at fixed diameter scale/dual complexes of degenerations (see, for instance, the results and conjectural pictures in \cite{Li25,Odaka22,spotti25,SZ2024}).

\printbibliography
\end{document}